\documentclass[12pt,reqno]{amsart}
\usepackage{amsmath, amsfonts, amssymb, amsthm, hyperref,cite}
\usepackage{bm}
\usepackage{mathtools}
\usepackage{mathrsfs}
\allowdisplaybreaks[4]
\def\l{\left}
\def\r{\right}
\def\bg{\bigg}
\def\({\bg(}
\def\){\bg)}
\def\t{\text}
\def\f{\frac}

\def\eq{\equiv}

\def\d{\t{d}}

\def\Z{\mathbb Z}

\def\N{\mathbb N}
\def\Q{\mathbb Q}

\def\<{\langle}
\def\>{\rangle}
\def\1{{\bf 1}}

\theoremstyle{plain}
\newtheorem{theorem}{Theorem}[section]

\newtheorem{lemma}{Lemma}[section]
\newtheorem{corollary}{Corollary}[section]
\theoremstyle{definition}

\newtheorem*{Acks}{Acknowledgments}
\theoremstyle{remark}

\newcommand{\CT}{\operatorname{CT}}

\numberwithin{equation}{section}

\begin{document}
\title[Congruences concerning $T_n(b,c)$]{Congruences concerning generalized central trinomial coefficients via the constant term method}

\author[Chen Wang]{Chen Wang$^{\ast}$}
\address{Department of Applied Mathematics, Nanjing Forestry University, Nanjing 210037, People's Republic of China}
\email{cwang@smail.nju.edu.cn}

\author[Xia-Han Cui]{Xia-Han Cui}
\address{Department of Applied Mathematics, Nanjing Forestry University, Nanjing 210037, People's Republic of China}
\email{xhcui@njfu.edu.cn}

\begin{abstract}
For any $n\in\mathbb{N}=\{0,1,2,\ldots\}$ and $b,c\in\mathbb{C}$, the $n$-th generalized central trinomial coefficient is defined as the constant term in the Laurent expansion of $(b+x+cx^{-1})^n$.  In this paper, utilizing the constant term method and generating functions, we prove several congruences concerning generalized central trinomial coefficients. First, we establish a parametric congruence modulo $p^2$, from which several conjectural congruences of Z.-W. Sun follow as special cases. We also prove a Catalan-type congruence modulo $p^2$ and a congruence involving $\binom{3k}{k}T_{3k}(6,1)$ modulo $p$, thereby confirming another two conjectural congruences of Z.-W. Sun.
\end{abstract}
\thanks{$^{\ast}$Corresponding author}
\subjclass[2020]{05A10, 05A19, 11A07, 11B65}
\keywords{The constant term method, congruence, generalized central trinomial coefficient, binomial coefficient}

\maketitle

\section{Introduction}

For any $n\in\N=\{0,1,2,\ldots\}$ and $b,c\in\mathbb C$, the $n$-th generalized central trinomial coefficient $T_n(b,c)$ (cf. \cite{Noe}) is defined by
$$
T_n(b,c)=[x^n](x^2+bx+c)^n.
$$
Equivalently, it is the constant term in the Laurent expansion of $(b+x+cx^{-1})^n$, namely
\begin{equation}\label{CTdef}
T_n(b,c)=\CT_x (b+x+cx^{-1})^n.
\end{equation}
By the multinomial theorem, we have
$$
T_n(b,c)
=
\sum_{j=0}^{\lfloor n/2\rfloor}
\binom{n}{2j}\binom{2j}{j}b^{n-2j}c^j.
$$
In particular, $T_n(1,1)=T_n$ is the classical central trinomial coefficient, while $T_n(2,1)=\binom{2n}{n}$ is the central binomial coefficient.

The coefficients $T_n(b,c)$ form a natural two-parameter extension of the central trinomial coefficients and have attracted considerable attention in combinatorics and number theory. They are closely related to Legendre polynomials (cf. \cite[p. 38]{Gould}). If $d=b^2-4c\ne0$, then one has
\begin{equation}\label{relation}
T_n(b,c)=(\sqrt d)^n P_n\left(\frac{b}{\sqrt d}\right),
\end{equation}
where $P_n(x)$ denotes the Legendre polynomial
\begin{equation}\label{legendre}
P_n(x)=
\sum_{j=0}^{n}
\binom{n}{j}\binom{n+j}{j}
\left(\frac{x-1}{2}\right)^j.
\end{equation}
This connection explains why congruences involving $T_n(b,c)$ are often related to congruences for Legendre polynomials (see, e.g., \cite{Guo2015,WangXia2021}). The generating function of $P_n(x)$ (cf. \cite[p. 188]{Wilf}) is given by
\begin{equation}\label{generatingfunc}
\sum_{n=0}^{\infty}P_n(x)t^n=\f{1}{\sqrt{1-2xt+t^2}}.
\end{equation}

It is known that sums involving products of binomial coefficients, such as
$$
\sum_{k=0}^{p-1} \frac{\binom{2k}{k}^2}{m^k},
\quad
\sum_{k=0}^{p-1} \frac{\binom{2k}{k}\binom{3k}{k}}{m^k},
\quad
\sum_{k=0}^{p-1} \frac{\binom{2k}{k}\binom{4k}{2k}}{m^k}
$$
frequently possess remarkable congruence properties modulo powers of primes (see, e.g., \cite{Guo2017,Liu2017,Mortenson2003a,Mortenson2003b,RV,SunZH2011,SunZH2020,SunZW2011,SunZW2013}). Since the generalized central trinomial coefficients include the central binomial coefficients as a special case, it is interesting to study congruences for sums involving products of binomial coefficients and $T_n(b,c)$. In 2014, Z.-W. Sun \cite{SunZW2014} initiated a systematic investigation of such congruences and proposed many conjectures concerning sums of the following forms
$$
\sum_{k=0}^{(p-1)/2}
\frac{\binom{2k}{k}T_{k}(b,c)}{m^k},
\quad
\sum_{k=0}^{(p-1)/2}
\frac{\binom{2k}{k}T_{2k}(b,c)}{m^k},
\quad
\sum_{k=0}^{(p-1)/2}
\frac{\binom{2k}{k}T_{2k}(b,c)}{(k+1)m^k},
\quad 
\sum_{k=0}^{(p-1)/2}
\frac{\binom{3k}{k}T_{3k}(b,c)}{m^k},
$$
where $b,c$ take special values.

Our first result is the following parametric congruence.

\begin{theorem}\label{mainthm1}
Let $p>3$ be a prime and let $b,c\in\mathbb Z$ with $p\nmid b+2c$. Then
\[
\sum_{k=0}^{(p-1)/2}
\frac{\binom{2k}{k}}{4^k(b+2c)^{2k}}
T_{2k}(b,c^2)
\equiv
(-\lambda)^{(p-1)/2}
+(-1)^{(p-1)/2} p
\sum_{\substack{0\leq j\leq p-1\\ j\ne {(p-1)/2}}}
\frac{\lambda^j}{2j+1}
\pmod{p^2},
\]
where
\[
\lambda=\frac{2c}{b+2c}.
\]
\end{theorem}

Taking suitable special values of $b$ and $c$, we obtain the following results which confirm some conjectures of Sun \cite[Conjecture 2.1]{SunZW2014}. 

\begin{corollary}\label{mainthm1cor}
For any prime $p>3$, we have
\begin{align}
\label{mainthm1coreq1}\sum_{k=0}^{(p-1)/2}\frac{\binom{2k}{k}}{16^k}T_{2k}(4,1)&\eq1\pmod{p^2},\\
\label{mainthm1coreq2}\sum_{k=0}^{(p-1)/2}\frac{\binom{2k}{k}}{4^k}T_{2k}(3,4)&\eq\l(\frac{-1}{p}\r)\f{7-3^p}{4}\pmod{p^2},\\
\label{mainthm1coreq3}\sum_{k=0}^{(p-1)/2}\frac{\binom{2k}{k}}{16^k}T_{2k}(8,9)&\eq\l(\f{3}{p}\r)\pmod{p^2},
\end{align}
where $(-)$ stands for the Legendre symbol.
\end{corollary}

We note that Sun \cite[Remark 2.1]{SunZW2014} also observed that for any prime $p>3$,
\begin{gather*}
\sum_{k=0}^{p-1}\f{\binom{2k}{k}}{4^k}T_{2k}(5,4)\eq1\pmod p,\quad \sum_{k=0}^{p-1}\f{\binom{2k}{k}}{4^k}T_{2k}(3,1)\eq\l(\f{2}{p}\r)\pmod p,\\
\sum_{k=0}^{p-1}\f{\binom{2k}{k}}{16^k}T_{2k}(4,9)\eq\l(\f{p}{3}\r)\pmod p,\quad \sum_{k=0}^{p-1}\f{\binom{2k}{k}}{16^k}T_{2k}(8,25)\eq\l(\f{-5}{p}\r)\pmod p.
\end{gather*}
By Theorem \ref{mainthm1} with $(b,c)=(5,-2),\, (3,-1),\, (4,-3),\, (8,-5)$ respectively, we obtain the modulus $p^2$ extensions of the above observations.

Recall that for $n\in\N$, the $n$-th Catalan number $\binom{2n}{n}/(n+1)\in\Z$. The following Catalan-type congruence similar to \eqref{mainthm1coreq1} confirms another conjecture of Sun \cite[Conjecture 2.1]{SunZW2014}.

\begin{theorem}\label{cat-main}
Let $p>3$ be a prime. Then
\[
\sum_{k=0}^{(p-1)/2}\frac{\binom{2k}{k}}{(k+1)16^k}T_{2k}(4,1)\equiv\frac43\left(\l(\frac3p\r)-p\l(\frac{-1}{p}\r)\right)\pmod {p^2}.
\]
\end{theorem}

Motivated by Theorem \ref{mainthm1}, it is natural to ask whether Theorem \ref{cat-main} admits a parametric generalization. More precisely, we are interested in determining the congruence
\[
\sum_{k=0}^{(p-1)/2}\frac{\binom{2k}{k}}{4^k(k+1)(b+2c)^{2k}}T_{2k}(b,c^2)\pmod{p^2}.
\]
This question appears to be rather difficult, and we leave it for future investigation.

Our third result is a modulus $p$ congruence concerning the $3k$-th generalized central trinomial coefficients, and it was also conjectured by Sun \cite[Conjecture 2.1]{SunZW2014}.

\begin{theorem}\label{triple-main}
Let $p>3$ be a prime. Then
\[
\sum_{k=0}^{p-1}\frac{\binom{3k}{k}}{432^k}T_{3k}(6,1)\equiv 1\pmod p.
\]
\end{theorem}

Our approach relies on the constant term method and generating functions in combinatorics. For a Laurent series $F(x)$, let $\operatorname{CT}_x F(x)$ denote its constant term, that is, the coefficient of $x^0$. The essence of this method is to realize combinatorial quantities as constant terms of Laurent polynomials or rational functions. This transformation enables us to convert finite sums involving binomial coefficients and generalized central trinomial coefficients into coefficient-extraction problems for one-variable algebraic functions. 

The rest of this paper is organized as follows. We shall prove Theorem \ref{mainthm1} and Corollary \ref{mainthm1cor} in Section \ref{sec:parametric}. Theorem \ref{cat-main} is proved in Section \ref{sec:catalan}. Finally, Section \ref{sec:triple} is devoted to the proof of Theorem \ref{triple-main}.

\section{Proofs of Theorem \ref{mainthm1} and Corollary \ref{mainthm1cor}}\label{sec:parametric}
For brevity, in the subsequent proofs, we always write $m=(p-1)/2$.

We first transform $T_n(b,c^2)$ into a symmetric form.
\begin{lemma}\label{symmetric_form}
For $n\in\N$ and $b,c\in\Z$, we have
$$
T_n(b,c^2)=\CT_x(b+c(x+x^{-1}))^n.
$$
\end{lemma}

\begin{proof}
From the definition of $T_n(b,c^2)$, we have
$$
T_n(b,c^2)=\CT_y(b+y+c^2y^{-1})^n.
$$
The lemma is clear when $c=0$. If $c\neq0$, then putting $y=cx$ and noting that $\CT_y=\CT_x$ we obtain
$$
T_n(b,c^2)=\CT_x(b+cx+cx^{-1})^n,
$$
as desired.
\end{proof}

The following two lemmas on generating functions play a crucial role in the proof of Lemma \ref{CTxQ}.

For $n\in\N$, define
\[
Q_n(z)=\sum_{k=0}^n(-1)^k\binom{n+k}{2k}z^{2k}.
\]

\begin{lemma}\label{gfunc}
We have
\[
\sum_{n=0}^\infty Q_n(z)t^n
=\frac{1-t}{(1-t)^2+tz^2}.
\]
\end{lemma}

\begin{proof}
Clearly,
\begin{align*}
\sum_{n=0}^\infty Q_n(z)t^n
&=\sum_{n=0}^\infty\sum_{k=0}^n
(-1)^k\binom{n+k}{2k}z^{2k}t^n  \\
&=\sum_{k=0}^\infty (-1)^k z^{2k}t^k
\sum_{r=0}^\infty \binom{r+2k}{2k}t^r,
\end{align*}
where $n=k+r$. Since
\[
\sum_{r=0}^\infty\binom{r+2k}{2k}t^r=\frac{1}{(1-t)^{2k+1}},
\]
we get
\[
\sum_{n=0}^\infty Q_n(z)t^n
=\frac{1}{1-t}
\sum_{k=0}^\infty
\left(-\frac{z^2t}{(1-t)^2}\right)^k
=\frac{1-t}{(1-t)^2+tz^2}.
\]
This concludes the proof.
\end{proof}

\begin{lemma}\label{CTid}
Let $\alpha(t)$ and $\beta(t)$ be polynomials in $t$ with coefficients over $\Q[i]$ such that $\alpha(0)\neq0$ and $\beta(0)=0$. Then
\[
\operatorname{CT}_x\frac{1}{\alpha(t)+\beta(t)(x+x^{-1})}
=
\frac{1}{\sqrt{\alpha(t)^2-4\beta(t)^2}},
\]
where the square root is the formal square root whose constant term
equals $\alpha(0)$.
\end{lemma}

\begin{proof}
Since $\alpha(0)\neq0$, $\alpha(t)$ is invertible. Therefore,
\begin{align*}
\CT_x\frac{1}{\alpha(t)+\beta(t)(x+x^{-1})}&=\CT_x\f{1}{\alpha(t)}\f{1}{1+\f{\beta(t)}{\alpha(t)}(x+x^{-1})}\\
&=\CT_x\sum_{k=0}^{\infty}(-1)^k\f{\beta(t)^k}{\alpha(t)^{k+1}}(x+x^{-1})^k.
\end{align*}
Note that
\begin{equation}\label{centralbinomial}
\CT_x(x+x^{-1})^k=\begin{cases}0,\quad& 2\nmid k,\\ \binom{k}{k/2},\quad& 2\mid k.\end{cases}
\end{equation}
So we have
$$
\CT_x\frac{1}{\alpha(t)+\beta(t)(x+x^{-1})}=\sum_{k=0}^{\infty}\f{\beta(t)^{2k}}{\alpha(t)^{2k+1}}\binom{2k}{k}=\frac{1}{\sqrt{\alpha(t)^2-4\beta(t)^2}}.
$$
This completes the proof.
\end{proof}

The following lemma establishes a relation between a certain constant term and a sum involving Legendre polynomials.

\begin{lemma}\label{CTxQ}
Let $b,c\in\Z$ with $b+2c\neq0$. Then, for $n\in\N$, we have
\[
\CT_xQ_n(u)=(-1)^n\sum_{k=0}^{2n}P_k\l(\f{b-2c}{b+2c}\r),
\]
where 
$$u=\f{2(b+c(x+x^{-1}))}{b+2c}.$$
\end{lemma}

\begin{proof}
By Lemma \ref{gfunc}, we have
\[
\sum_{n=0}^\infty
\CT_x Q_n(u)t^{2n}
=
\CT_x
\frac{1-t^2}{(1-t^2)^2+t^2u^2}.
\]
It is easy to see that
\[
\frac{1-t^2}{(1-t^2)^2+t^2u^2}
=\frac12\left(
\frac{1}{1-t^2+itu}
+
\frac{1}{1-t^2-itu}
\right).
\]
Note that
$$
1-t^2+itu=\l(1-t^2+\f{2ibt}{b+2c}\r)+\f{2ict}{b+2c}(x+x^{-1}).
$$
With the help of Lemma \ref{CTid}, we obtain
\begin{align*}
\CT_x\frac{1}{1-t^2+itu}&=\f{1}{\sqrt{\l(1-t^2+\f{2ibt}{b+2c}\r)^2+\f{16c^2t^2}{(b+2c)^2}}}\\
&=\f{1}{\sqrt{(1-t^2)^2-\f{4b^2t^2}{(b+2c)^2}+\f{4ibt(1-t^2)}{b+2c}+\f{16c^2t^2}{(b+2c)^2}}}\\
&=\f{1}{(1+it)\sqrt{1+\f{b-2c}{b+2c}2it-t^2}}.
\end{align*}
Similarly, 
$$
\CT_x\frac{1}{1-t^2-itu}=\f{1}{(1-it)\sqrt{1-\f{b-2c}{b+2c}2it-t^2}}.
$$
Combining the above, we arrive at
\begin{align}\label{CTxQkey1}
\sum_{n=0}^\infty
\CT_x Q_n(u)t^{2n}&=\f12\l(\CT_x\frac{1}{1-t^2+itu}+\CT_x\frac{1}{1-t^2-itu}\r)\notag\\
&=\f12\l(\f{1}{(1+it)\sqrt{1+\f{b-2c}{b+2c}2it-t^2}}+\f{1}{(1-it)\sqrt{1-\f{b-2c}{b+2c}2it-t^2}}\r).
\end{align}

On the other hand, by \eqref{generatingfunc} we have
\begin{align}\label{CTxQkey2}
&\sum_{n=0}^{\infty}(-1)^n\sum_{k=0}^{2n}P_k\l(\f{b-2c}{b+2c}\r)t^{2n}\notag\\
&\qquad=\f12\sum_{n=0}^{\infty}\sum_{k=0}^{n}P_k\l(\f{b-2c}{b+2c}\r)(-it)^n+\f12\sum_{n=0}^{\infty}\sum_{k=0}^{n}P_k\l(\f{b-2c}{b+2c}\r)(it)^n\notag\\
&\qquad=\f12\sum_{k=0}^{\infty}P_k\l(\f{b-2c}{b+2c}\r)(-it)^k\sum_{n=k}^{\infty}(-it)^{n-k}+\f12\sum_{k=0}^{\infty}P_k\l(\f{b-2c}{b+2c}\r)(it)^k\sum_{n=k}^{\infty}(it)^{n-k}\notag\\
&\qquad=\f12\l(\f{1}{(1+it)\sqrt{1+\f{b-2c}{b+2c}2it-t^2}}+\f{1}{(1-it)\sqrt{1-\f{b-2c}{b+2c}2it-t^2}}\r).
\end{align}
Combining \eqref{CTxQkey1} and \eqref{CTxQkey2}, we get
$$
\sum_{n=0}^\infty
\CT_x Q_n(u)t^{2n}=\sum_{n=0}^{\infty}(-1)^n\sum_{k=0}^{2n}P_k\l(\f{b-2c}{b+2c}\r)t^{2n}.
$$
Finally, comparing coefficients of $t^{2n}$ proves the desired lemma. 
\end{proof}

\begin{lemma}\label{legendrecon}
Under the conditions of Theorem \ref{mainthm1}, we have
\[
\sum_{k=0}^{p-1}P_k\l(\f{b-2c}{b+2c}\r)\eq \lambda^{(p-1)/2}
+ p
\sum_{\substack{0\leq j\leq p-1\\ j\ne {(p-1)/2}}}
\frac{\lambda^j}{2j+1}
\pmod{p^2},
\]
where $\lambda$ is defined as in Theorem \ref{mainthm1}.
\end{lemma}

\begin{proof}
By \eqref{legendre},
\begin{align*}
\sum_{k=0}^{p-1}P_k\l(\f{b-2c}{b+2c}\r)&=\sum_{k=0}^{p-1}\sum_{j=0}^k\binom{k}{j}\binom{k+j}{j}(-\lambda)^j\\
&=\sum_{j=0}^{p-1}\binom{2j}{j}(-\lambda)^j\sum_{k=j}^{p-1}\binom{k+j}{2j}\\
&=\sum_{j=0}^{p-1}\binom{p+j}{2j+1}\binom{2j}{j}(-\lambda)^j,
\end{align*}
where in the last step we have used Chu's identity. For $0\leq j\leq p-1$, one has
\[
\binom{2j}{j}\binom{p+j}{2j+1}
=
\frac{p}{2j+1}
\prod_{r=1}^{j}\frac{p^2-r^2}{r^2}\eq\f{p(-1)^j}{2j+1}\pmod{p^2}.
\]
Combining the above, we conclude the proof.
\end{proof}

\medskip

\noindent{\it Proof of Theorem \ref{mainthm1}}. By \cite[Lemma 3.1]{SunZW2011}, for $0\leq k\leq m$, we have
\begin{equation}\label{binomialcon}
\binom{2k}{k}\eq\binom{m+k}{2k}(-16)^k\pmod{p^2}.
\end{equation}
Therefore, in view of Lemmas \ref{symmetric_form} and \ref{CTxQ}, we have
\begin{align*}
\sum_{k=0}^{m}\frac{\binom{2k}{k}}{4^k(b+2c)^{2k}}T_{2k}(b,c^2)&\eq\CT_x\sum_{k=0}^{m}\f{\binom{m+k}{2k}(-4)^k}{(b+2c)^{2k}}(b+c(x+x^{-1}))^{2k}\\
&=\CT_xQ_m\l(\f{2(b+c(x+x^{-1}))}{b+2c}\r)\\
&=(-1)^m\sum_{k=0}^{p-1}P_k\l(\f{b-2c}{b+2c}\r)\pmod{p^2}.
\end{align*}
Then the desired result follows immediately from Lemma \ref{legendrecon}.\qed

\medskip

Taking $(b,c)=(4,-1),\, (3,-2),\,(8,-3)$ in Theorem \ref{mainthm1}, we obtain
\begin{align}
\label{corkey1}\sum_{k=0}^{(p-1)/2}\frac{\binom{2k}{k}}{16^k}T_{2k}(4,1)&\eq1+(-1)^{(p-1)/2} p
\sum_{\substack{0\leq j\leq p-1\\ j\ne {(p-1)/2}}}
\frac{(-1)^j}{2j+1}\pmod{p^2},\\
\label{corkey2}\sum_{k=0}^{(p-1)/2}\frac{\binom{2k}{k}}{4^k}T_{2k}(3,4)&\eq(-4)^{(p-1)/2}+(-1)^{(p-1)/2} p
\sum_{\substack{0\leq j\leq p-1\\ j\ne {(p-1)/2}}}
\frac{4^j}{2j+1}\pmod{p^2},\\
\label{corkey3}\sum_{k=0}^{(p-1)/2}\frac{\binom{2k}{k}}{16^k}T_{2k}(8,9)&\eq3^{(p-1)/2}+(-1)^{(p-1)/2} p
\sum_{\substack{0\leq j\leq p-1\\ j\ne {(p-1)/2}}}
\frac{(-3)^j}{2j+1}\pmod{p^2}.
\end{align}
To prove Corollary \ref{mainthm1cor}, we need the following lemma.

\begin{lemma}\label{corsumkey}
For any prime $p>3$, we have
\begin{align}
\label{harmonic1}\sum_{\substack{0\leq j\leq p-1\\ j\ne {(p-1)/2}}}
\frac{(-1)^j}{2j+1}&\eq0\pmod{p},\\
\label{harmonic2}\sum_{\substack{0\leq j\leq p-1\\ j\ne {(p-1)/2}}}
\frac{4^j}{2j+1}&\eq -q_p(2)-\frac34 q_p(3)\pmod{p},\\
\label{harmonic3}\sum_{\substack{0\leq j\leq p-1\\ j\ne {(p-1)/2}}}
\frac{(-3)^j}{2j+1}&\eq-\frac12\left(\frac{-3}{p}\right)q_p(3)\pmod{p},
\end{align}
where $q_p(a)=(a^{p-1}-1)/p$ stands for the Fermat quotient with respect to the $p$-adic unit $a$.
\end{lemma}

\begin{proof}
For any $t\in\Z$ with $p\nmid t$, we have
\begin{align}\label{generaltrans}
\sum_{\substack{0\leq j\leq p-1\\ j\ne {(p-1)/2}}}
\frac{t^j}{2j+1}&=\sum_{j=0}^{(p-3)/2}\f{t^j}{2j+1}+\sum_{j=(p+1)/2}^{p-1}\f{t^j}{2j+1}\notag\\
&=t^{(p-1)/2}\sum_{j=1}^{(p-1)/2}\f{t^{-j}}{2((p-1)/2-j)+1}+t^{(p-1)/2}\sum_{j=1}^{(p-1)/2}\f{t^j}{2((p-1)/2+j)+1}\notag\\
&\eq \f{t^{(p-1)/2}}{2}\sum_{j=1}^{(p-1)/2}\l(\f{t^j}{j}-\f{t^{-j}}{j}\r)\pmod{p}.
\end{align}
Putting $t=-1$ in \eqref{generaltrans}, we obtain \eqref{harmonic1} at once.

Write $t=u^2$. If $t$ is not a square, the following congruences are understood in $\mathbb Z_p[u]$. Note that
\begin{align*}
\sum_{j=1}^{p-1}\f{u^j}{j}&=\sum_{j=1}^{(p-1)/2}\f{u^{2j}}{2j}+\sum_{j=0}^{(p-3)/2}\f{u^{2j+1}}{2j+1}\\
&=\f{1}{2}\sum_{j=1}^{(p-1)/2}\f{t^j}{j}+\sum_{j=1}^{(p-1)/2}\f{u^{2((p-1)/2-j)+1}}{2((p-1)/2-j)+1}\\
&\eq\f12\sum_{j=1}^{(p-1)/2}\f{t^j}{j}-\f{u^p}{2}\sum_{j=1}^{(p-1)/2}\f{t^{-j}}{j}\pmod{p}.
\end{align*}
Similarly,
$$
\sum_{j=1}^{p-1}\f{(-u)^j}{j}\eq \f12\sum_{j=1}^{(p-1)/2}\f{t^{j}}{j}+\f{u^{p}}{2}\sum_{j=1}^{(p-1)/2}\f{t^{-j}}{j}\pmod{p}.
$$
Therefore,
\begin{equation}\label{keykey}
\sum_{j=1}^{(p-1)/2}\l(\f{t^j}{j}-\f{t^{-j}}{j}\r)\eq(u^{-p}+1)\sum_{j=1}^{p-1}\f{u^j}{j}+(1-u^{-p})\sum_{j=1}^{p-1}\f{(-u)^{j}}{j}\pmod{p}.
\end{equation}
In 2004, Granville \cite{Granville} showed that
\begin{equation}\label{Granvilleeq}
\sum_{j=1}^{p-1}\f{z^j}{j}\eq\f{1-z^p-(1-z)^p}{p}\pmod{p}.
\end{equation}

Now we prove \eqref{harmonic2}. Since $t=4=2^2$, in view of \eqref{keykey} and \eqref{Granvilleeq} we obtain
\begin{align*}
\f{2^{p-1}}{2}\sum_{j=1}^{(p-1)/2}\l(\f{4^j}{j}-\f{4^{-j}}{j}\r)&\eq\f12\l(\f32\f{1-2^p-(-1)^p}{p}+\f12\f{1+2^p-3^p}{p}\r)\\
&=-q_p(2)-\f{3}{4}q_p(3)\pmod{p}.
\end{align*}
This, together with \eqref{generaltrans}, gives \eqref{harmonic2}.

Below we consider \eqref{harmonic3}. In this case, $t=-3=(\sqrt{-3})^2=(2\omega+1)^2$, where $\omega$ is a primitive cubic root of unit. Using \eqref{keykey} and \eqref{Granvilleeq} we have
\begin{align}\label{sqrt-3}
&\f{(-3)^{(p-1)/2}}{2}\sum_{j=1}^{(p-1)/2}\l(\f{(-3)^j}{j}-\f{(-3)^{-j}}{j}\r)\notag\\
&\qquad\eq\f{(-3)^{(p-1)/2}}{2}((\sqrt{-3})^{-p}+1)\f{1-(\sqrt{-3})^p-(-2\omega)^p}{p}\notag\\
&\qquad\quad+\f{(-3)^{(p-1)/2}}{2}(1-(\sqrt{-3})^{-p})\f{1+(\sqrt{-3})^p-(-2\omega^2)^p}{p}\pmod{p}.
\end{align}
Note that modulo $p$,
$$
(-3)^{(p-1)/2}\eq\l(\f{-3}{p}\r)=\begin{cases}1,\quad&p\eq1\pmod3,\\ -1,\quad&p\eq2\pmod3.\end{cases}
$$
Thus we can write 
$$
(-3)^{(p-1)/2}=\l(\f{-3}{p}\r)+ps,
$$
where $s$ is a $p$-adic integer. If $p\eq1\pmod3$, then the right-hand side of \eqref{sqrt-3} equals
\begin{align*}
&\f{(-3)^{(p-1)/2}}{2}\f{2+2^p\omega+2^p\omega^2}{p}+\f{(-3)^{(p-1)/2}}{2}(\sqrt{-3})^{-p}\f{-2(\sqrt{-3})^{p}+2^p\omega-2^p\omega^2}{p}\\
&\qquad\eq\f{1-(-3)^{(p-1)/2}}{p}=-s\pmod{p}.
\end{align*}
Also, if $p\eq2\pmod3$, the right-hand side of \eqref{sqrt-3} equals $-s$. Thus, it remains to show 
\begin{equation}\label{vcon}
s\eq\f12\l(\f{-3}{p}\r)q_p(3)\pmod{p}.    
\end{equation}
In fact,
\begin{align*}
1+pq_p(3)=(-3)^{p-1}=\l(\l(\f{-3}{p}\r)+ps\r)^2\eq 1+2p\l(\f{-3}{p}\r)s\pmod{p^2}.
\end{align*}
This means
$$
q_p(3)\eq2\l(\f{-3}{p}\r)s\pmod{p},
$$
which is equivalent to \eqref{vcon}. This proves \eqref{harmonic3}.

The proof of Lemma \ref{corsumkey} is now complete.
\end{proof}

\medskip

\noindent{\it Proof of Corollary \ref{mainthm1cor}}. Substituting \eqref{harmonic1} into \eqref{corkey1}, we obtain \eqref{mainthm1coreq1}.

Combining \eqref{corkey2} and \eqref{harmonic2}, we have
$$
\sum_{k=0}^{(p-1)/2}\frac{\binom{2k}{k}}{4^k}T_{2k}(3,4)\eq\l(\f{-1}{p}\r)\l(1-\f{3}{4}pq_p(3)\r)=\l(\f{-1}{p}\r)\f{7-3^p}{4}\pmod{p^2},
$$
which proves \eqref{mainthm1coreq2}.

By \eqref{corkey3} and \eqref{harmonic3}, we deduce that
$$
\sum_{k=0}^{(p-1)/2}\frac{\binom{2k}{k}}{16^k}T_{2k}(8,9)\eq3^{(p-1)/2}-\f{p}{2}\l(\f{3}{p}\r)q_p(3)\pmod{p^2}.
$$
In view of \eqref{vcon},
$$
3^{(p-1)/2}=\l(\f{-1}{p}\r)(-3)^{(p-1)/2}\eq \l(\f{3}{p}\r)+\f{p}{2}\l(\f{3}{p}\r)q_p(3)\pmod{p^2}.
$$
Combining the above two congruences, we arrive at \eqref{mainthm1coreq3}.

The proof of Corollary \ref{mainthm1cor} is now complete.\qed

\section{Proof of Theorem \ref{cat-main}}\label{sec:catalan}

Throughout this section, write $v=4+x+x^{-1}$. For $n\in \N$, define
$$
R_n(z)=\sum_{k=0}^n\f{(-1)^k}{k+1}\binom{n+k}{2k}z^{2k}.
$$

We need a usable closed form for $R_n(z)$.

\begin{lemma}\label{Rn(z)}
For $n\in\Z^+$, we have
$$
R_n(z)=\f{1}{z^2}\l(\f{2-V_{n+1}(z^2)}{n+1}-\f{2-V_n(z^2)}{n}\r),
$$
where $V_k(z^2)=\alpha^k+\beta^k$ and $\alpha,\beta$ are the two roots of 
$$
t^2-(2-z^2)t+1=0.
$$
\end{lemma}

\begin{proof}
It is easy to see that
\begin{align*}
\sum_{n=0}^{\infty}R_n(z)t^n
&=\sum_{n=0}^{\infty}\sum_{k=0}^n\f{(-1)^k}{k+1}\binom{n+k}{2k}z^{2k}t^n \\
&=\sum_{k=0}^{\infty}\f{(-1)^kz^{2k}}{k+1}\sum_{n=k}^{\infty}\binom{n+k}{2k}t^n\\
&=\f{1}{1-t}\sum_{k=0}^{\infty}\f{1}{k+1}\l(-\f{z^2t}{(1-t)^2}\r)^k\\
&=\frac{1-t}{z^2t}\log\left(1+\frac{z^2t}{(1-t)^2}\right).
\end{align*}
Now
\[
1+\frac{z^2t}{(1-t)^2}=\frac{(1-\alpha t)(1-\beta t)}{(1-t)^2},
\]
because $\alpha+\beta=2-z^2$ and $\alpha\beta=1$. Hence
\[
\log\left(1+\frac{z^2t}{(1-t)^2}\right)=\sum_{k=1}^{\infty}\frac{2-V_k(z^2)}{k}t^k.
\]
Multiplying by $(1-t)/(z^2t)$ and comparing the coefficient of $t^n$ gives the desired formula.
\end{proof}

Now we introduce an auxiliary quantity $y$ and write $v=y+y^{-1}$. Define 
$$
d_n=\CT_x\f{y^n+y^{-n}}{v^2}\quad(n\in\N).
$$ 
Since $y^n+y^{-n}$ is a polynomial in $v$, the above definition is independent of the choice of $y$. 

The following result gives a simplification of $\CT_xR_n(v)$.

\begin{lemma}\label{CTxR}
For $n\in\Z^+$, we have
$$
\CT_xR_n(v)=(-1)^n\l(\f{d_{2n+2}}{n+1}+\f{d_{2n}}{n}\r).
$$
\end{lemma}

\begin{proof}
Putting $z=v$ in Lemma \ref{Rn(z)}, we have
$$
2-v^2=2-(y+y^{-1})^2=-(y^2+y^{-2}),
$$
and then the two roots of $t^2-(2-z^2)t+1=0$ are $\alpha=-y^2$ and $\beta=-y^{-2}$. Thus 
$$
V_k(v^2)=(-1)^k(y^{2k}+y^{-2k}),
$$
and
$$
R_n(v)=\f{1}{v^2}\l(\f{2-(-1)^{n+1}(y^{2n+2}+y^{-2n-2})}{n+1}-\f{2-(-1)^n(y^{2n}+y^{-2n})}{n}\r).
$$
Noting that
$$
\CT_x\f{1}{v^2}=\CT_x\f{x^2}{(x^2+4x+1)^2}=0,
$$
we arrive at
$$
\CT_xR_n(v)=(-1)^n\l(\CT_x\f{y^{2n+2}+y^{-2n-2}}{(n+1)v^2}+\CT_x\f{y^{2n}+y^{-2n}}{nv^2}\r)=(-1)^n\l(\f{d_{2n+2}}{n+1}+\f{d_{2n}}{n}\r)
$$
as desired.
\end{proof}

Now we establish the generating function of $d_n$.

\begin{lemma}\label{dngfunc}
We have
$$
\sum_{n=0}^{\infty}d_nt^n=\f{t^2(1+t)}{(1+t^2)^2\sqrt{1-6t+t^2}}.
$$
\end{lemma}

\begin{proof}
Since $y+y^{-1}=v$,
$$
\sum_{n=0}^{\infty}(y^n+y^{-n})t^n=\f{1}{1-yt}+\f{1}{1-y^{-1}t}=\f{2-(y+y^{-1})t}{1-(y+y^{-1})t+t^2}=\f{2-vt}{1-vt+t^2}.
$$
Therefore, via the partial fraction decomposition,
\begin{align*}
\sum_{n=0}^{\infty}d_nt^n&=\CT_x\f{2-vt}{v^2(1-vt+t^2)}\\
&=\CT_x\f{2}{(t^2+1)v^2}-\CT_x\f{t(t^2-1)}{(t^2+1)^2v}+\CT_x\f{t^2(t^2-1)}{(t^2+1)^2(tv-(t^2+1))}.
\end{align*}
Note that $\CT_xv^{-1}=\CT_xv^{-2}=0$. Then we have
$$
\sum_{n=0}^{\infty}d_nt^n=\CT_x\f{t^2(t^2-1)}{(t^2+1)^2(tv-(t^2+1))}=\f{t^2(1-t^2)}{(1+t^2)^2}\CT_x\f{1}{1+t^2-4t-t(x+x^{-1})}.
$$
In view of Lemma \ref{CTid}, we obtain the desired result.
\end{proof}

The following lemmas are devoted to determining $d_{p+1}+d_{p-1}\pmod{p}$ and $d_{p+1}-d_{p-1}\pmod{p^2}$.

\begin{lemma}\label{tNf-rev}
Let $N$ be a positive integer and let $f(t)$ be a integer-coefficient polynomial in $t$ with $\deg f(t)\leq N+1$. Then
$$
[t^N]\f{f(t)}{1+t^2}=\begin{cases}\f{(-1)^{\lfloor N/2\rfloor}}{2}\l(f(i)+f(-i)\r),\quad&2\mid N,\\[5pt] \f{(-1)^{\lfloor N/2\rfloor}}{2i}\l(f(i)-f(-i)\r),\quad&2\nmid N.\end{cases}
$$
\end{lemma}

\begin{proof}
Divide $f(t)$ by $1+t^2$. The quotient has degree at most $N-1$, so it contributes
nothing to the coefficient of $t^N$.  Write the remainder as $qt+r$. 

If $N$ is even, write $N=2n$. Then
\[
[t^N]\frac{qt+r}{1+t^2}=[t^{2n}]\sum_{k=0}^{\infty}(-1)^k t^{2k}(qt+r)=[t^{2n}]\sum_{k=0}^{\infty}(-1)^kr t^{2k}=(-1)^nr.
\]
On the other hand, since $f(\pm i)=\pm qi+r$, we have
\[
r=\frac{f(i)+f(-i)}2.
\]
Therefore,
\[
[t^N]\frac{qt+r}{1+t^2}=\f{(-1)^n}{2}\l(f(i)+f(-i)\r).
\]

If $N$ is odd, write $N=2n+1$. Then
\[
[t^N]\frac{qt+r}{1+t^2}=[t^{2n+1}]\sum_{k=0}^{\infty}(-1)^k t^{2k}(qt+r)=[t^{2n+1}]\sum_{k=0}^{\infty}(-1)^kq t^{2k+1}=(-1)^nq.
\]
Since
$$
q=\frac{f(i)-f(-i)}{2i},
$$
we have
\[
[t^N]\frac{qt+r}{1+t^2}=\f{(-1)^n}{2i}\l(f(i)-f(-i)\r).
\]

\end{proof}

\begin{lemma}\label{lem:dplus}
For each prime $p>3$,
\[
[t^{p-1}]\frac{1+t}{1+t^2}\frac1{\sqrt{1-6t+t^2}}\equiv \l(\f{3}{p}\r)\pmod p.
\]
\end{lemma}

\begin{proof}
It is easy to see that
\begin{align*}
[t^{p-1}]\frac{1+t}{1+t^2}\frac1{\sqrt{1-6t+t^2}}&=[t^{p-1}]\frac{1+t}{1+t^2}\sum_{k=0}^{\infty}\binom{-1/2}{k}(-6t+t^2)^k\\
&=[t^{p-1}]\frac{1+t}{1+t^2}\sum_{k=0}^{p-1}\binom{-1/2}{k}(-6t+t^2)^k\\
&\eq [t^{p-1}]\frac{1+t}{1+t^2}\sum_{k=0}^{(p-1)/2}\binom{(p-1)/2}{k}(-6t+t^2)^k\\
&=[t^{p-1}]\frac{(1+t)(1-6t+t^2)^{(p-1)/2}}{1+t^2}\pmod{p}.
\end{align*}
Note that $\deg((1+t)(1-6t+t^2)^{(p-1)/2})=p$. Then, by Lemma \ref{tNf-rev},
\begin{align*}
[t^{p-1}]\frac{1+t}{1+t^2}\frac1{\sqrt{1-6t+t^2}}&\eq\frac{(-1)^{(p-1)/2}}{2}\left((1+i)(-6i)^{(p-1)/2}+(1-i)(6i)^{(p-1)/2}\right)\\
&\eq \begin{cases}(-1)^{(p-1)/4+(p^2-1)/8}\l(\f{3}{p}\r)\pmod{p},\quad& p\eq1\pmod4,\\ (-1)^{(p+1)/4+(p^2-1)/8}\l(\f{3}{p}\r)\pmod{p},\quad& p\eq3\pmod4.\end{cases}
\end{align*}
Moreover, one can directly verify that when $p\eq1\pmod4$, $(p-1)/4+(p^2-1)/8\eq0\pmod2$ and when $p\eq3\pmod4$, $(p+1)/4+(p^2-1)/8\eq0\pmod2$.

The proof of Lemma \ref{lem:dplus} is now complete.
\end{proof}

\begin{lemma}\label{lem:dminuskey}
For each prime $p>3$,
\[
[t^{p-1}]\frac{1-t}{(1+t^2)\sqrt{1-6t+t^2}}\equiv \left(\frac{-3}{p}\right)\pmod{p^2}.
\]
\end{lemma}

\begin{proof}
Let 
$$
g(z)=\sum_{k=0}^{\infty}\f{z^{k+1}}{k+1}\sum_{l=0}^{k}\binom{2l}{l}.
$$
We first show that
\begin{equation}\label{dminuskey-key}
\f{1}{2}\f{\d}{\d t}g\l(\f{2t}{(1+t)^2}\r)=\f{1-t}{(1+t^2)\sqrt{1-6t+t^2}}.
\end{equation}
In fact, the left-hand side of \eqref{dminuskey-key} equals
\begin{align*}
\f{1}{2}g'\l(\f{2t}{(1+t)^2}\r)\f{\d}{\d t}\l(\f{2t}{(1+t)^2}\r).
\end{align*}
Note that
\begin{align*}
g'(z)&=\sum_{k=0}^{\infty}z^k\sum_{l=0}^k\binom{2l}{l}=\sum_{l=0}^{\infty}\binom{2l}{l}\sum_{k=l}^{\infty}z^k\\
&=\f{1}{1-z}\sum_{l=0}^{\infty}\binom{2l}{l}z^l=\f{1}{(1-z)\sqrt{1-4z}}
\end{align*}
and
$$
\f{\d}{\d t}\l(\f{2t}{(1+t)^2}\r)=\f{2(1-t)}{(1+t)^3}.
$$
Therefore,
\begin{align*}
\f{1}{2}g'\l(\f{2t}{(1+t)^2}\r)\f{\d}{\d t}\l(\f{2t}{(1+t)^2}\r)=\f12\f{1}{(1-\f{2t}{(1+t)^2})\sqrt{1-\f{8t}{(1+t)^2}}}\f{2(1-t)}{(1+t)^3}=\f{1-t}{(1+t^2)\sqrt{1-6t+t^2}},
\end{align*}
which proves \eqref{dminuskey-key}.

Now, by \eqref{dminuskey-key},
\begin{align*}
[t^{p-1}]\frac{1-t}{(1+t^2)\sqrt{1-6t+t^2}}&=\f12[t^{p-1}]\f{\d}{\d t}g\l(\f{2t}{(1+t)^2}\r)\\
&=\f{p}{2}[t^p]g\l(\f{2t}{(1+t)^2}\r)\\
&=\f{p}{2}[t^p]\sum_{k=0}^{\infty}\f{1}{k+1}\f{2^{k+1}t^{k+1}}{(1+t)^{2(k+1)}}\sum_{l=0}^{k}\binom{2l}{l}\\
&=\f{p}{2}[t^p]\sum_{k=0}^{\infty}\f{2^{k+1}}{k+1}\sum_{j=0}^{\infty}\binom{-2(k+1)}{j}t^{j+k+1}\sum_{l=0}^{k}\binom{2l}{l}\\
&=\f{p}{2}\sum_{k=0}^{p-1}\f{2^{k+1}}{k+1}\binom{-2(k+1)}{p-k-1}\sum_{l=0}^k\binom{2l}{l}.
\end{align*}
Observe that for $0\leq k\leq (p-3)/2$,
$$
\binom{-2(k+1)}{p-k-1}\eq\binom{p-2k-2}{p-k-1}\eq0\pmod{p},
$$
and then
$$
\f{p}{2}\sum_{k=0}^{(p-3)/2}\f{2^{k+1}}{k+1}\binom{-2(k+1)}{p-k-1}\sum_{l=0}^k\binom{2l}{l}\eq0\pmod{p^2}.
$$
Therefore,
\begin{align*}
[t^{p-1}]\frac{1-t}{(1+t^2)\sqrt{1-6t+t^2}}&\eq\f{p}{2}\sum_{k=(p-1)/2}^{p-2}\f{2^{k+1}}{k+1}\binom{-2(k+1)}{p-k-1}\sum_{l=0}^k\binom{2l}{l}+2^{p-1}\sum_{l=0}^{p-1}\binom{2l}{l}\\
&=\f{p}{2}\sum_{k=1}^{(p-1)/2}\f{2^{p-k}}{p-k}\binom{-2(p-k)}{k}\sum_{l=0}^{p-1-k}\binom{2l}{l}+2^{p-1}\sum_{l=0}^{p-1}\binom{2l}{l}\\
&\eq 2^{p-1}\sum_{l=0}^{p-1}\binom{2l}{l}-p\sum_{l=0}^{p-1}\binom{2l}{l}\sum_{k=1}^{(p-1)/2}\f{1}{k2^k}\binom{2k}{k}\pmod{p^2}.
\end{align*}
Recall that in 2010--2011, Sun and Tauraso \cite{ST2010,ST2011} showed that for any odd prime $p$,
\begin{align*}
\sum_{l=0}^{p-1}\binom{2l}{l}&\eq\l(\f{p}{3}\r)\pmod{p^2},\\
\sum_{k=1}^{(p-1)/2}\f{1}{k2^k}\binom{2k}{k}&\eq q_p(2)\pmod{p}.
\end{align*}
Hence we have
$$
[t^{p-1}]\frac{1-t}{(1+t^2)\sqrt{1-6t+t^2}}\eq \l(\f{p}{3}\r)\l(2^{p-1}-pq_p(2)\r)=\l(\f{p}{3}\r)\pmod{p^2}.
$$
Then the desired result follows from the fact
$$
\l(\f{p}{3}\r)\l(\f{-3}{p}\r)=(-1)^{(p-1)/2}(-1)^{(p-1)/2}=1.
$$
This completes the proof.
\end{proof}

\begin{lemma}\label{lem:dminus}
For each prime $p>3$,
\[
[t^{p-1}]\frac{(1+t)(1-t^2)}{(1+t^2)^2}\frac1{\sqrt{1-6t+t^2}}\equiv\frac23\l(\f{-3}{p}\r)+p\left(\l(\f{3}{p}\r)-\frac23\right)\pmod {p^2}.
\]
\end{lemma}

\begin{proof}
We can directly verify that
$$
\frac{(1+t)(1-t^2)}{(1+t^2)^2\sqrt{1-6t+t^2}}=\frac23\frac{1-t}{(1+t^2)\sqrt{1-6t+t^2}}-\frac16\frac{\d}{\d t}\left(\frac{(1+t)\sqrt{1-6t+t^2}}{1+t^2}\right).
$$
It follows that
\begin{equation}\label{dminus-key}
[t^{p-1}]\frac{(1+t)(1-t^2)}{(1+t^2)^2\sqrt{1-6t+t^2}}=\frac23[t^{p-1}]\frac{1-t}{(1+t^2)\sqrt{1-6t+t^2}}-\frac{p}{6}[t^p]\frac{(1+t)\sqrt{1-6t+t^2}}{1+t^2}.
\end{equation}

Clearly,
\begin{align}\label{dminus-2parts}
[t^p]\frac{(1+t)\sqrt{1-6t+t^2}}{1+t^2}&=[t^p]\frac{(1+t)}{1+t^2}\sum_{k=0}^{p}\binom{1/2}{k}(-6t+t^2)^k\notag\\
&=-\binom{1/2}{p}6^p+[t^p]\frac{(1+t)}{1+t^2}\sum_{k=0}^{p-1}\binom{1/2}{k}(-6t+t^2)^k.
\end{align}
By Lucas' congruence,
$$
\binom{1/2}{p}6^p\eq\f{1/2}{1/2-p}\binom{-1/2}{p}6\eq 6\f{\binom{2p}{p}}{(-4)^p}\eq -3\pmod{p}.
$$
Meanwhile,
\begin{align*}
[t^p]\frac{(1+t)}{1+t^2}\sum_{k=0}^{p-1}\binom{1/2}{k}(-6t+t^2)^k&\eq[t^p]\frac{(1+t)}{1+t^2}\sum_{k=0}^{(p+1)/2}\binom{(p+1)/2}{k}(-6t+t^2)^k\\
&=[t^p]\frac{(1+t)(1-6t+t^2)^{(p+1)/2}}{1+t^2}\pmod{p}.
\end{align*}
Write 
$$
(1+t)(1-6t+t^2)^{(p+1)/2}=(1+t^2)h(t)+qt+r,
$$
where $h(t)$ is a monic integer-coefficient polynomial with $\deg h(t)=p$ and $q,r\in\Z$. Then, by Lemma \ref{tNf-rev}, we have
\begin{align*}
[t^p]\frac{(1+t)(1-6t+t^2)^{(p+1)/2}}{1+t^2}&=[t^p]h(t)+[t^p]\f{qt+r}{1+t^2}=1+[t^p]\f{qt+r}{1+t^2}=1+(-1)^{(p-1)/2}q\\
&=1+(-1)^{(p-1)/2}\f{(1+i)(-6i)^{(p+1)/2}-(1-i)(6i)^{(p+1)/2}}{2i}\\
&=\begin{cases}1-6^{(p+1)/2}(-1)^{(p-1)/4},\quad& p\eq1\pmod4,\\[5pt] 1-6^{(p+1)/2}(-1)^{(p+1)/4},\quad& p\eq3\pmod4\end{cases}\\
&\eq 1-6\l(\f{3}{p}\r)\pmod{p}.
\end{align*}
Substituting these into \eqref{dminus-2parts}, we have
$$
[t^p]\frac{(1+t)\sqrt{1-6t+t^2}}{1+t^2}\eq4-6\l(\f{3}{p}\r)\pmod{p}.
$$
Combining this, Lemma \ref{lem:dminuskey} and \eqref{dminus-key}, we finally arrive at
\begin{align*}
[t^{p-1}]\frac{(1+t)(1-t^2)}{(1+t^2)^2\sqrt{1-6t+t^2}}&\eq\f23\l(\f{-3}{p}\r)-\f{p}{6}\l(4-6\l(\f{3}{p}\r)\r)\\
&=\frac23\l(\f{-3}{p}\r)+p\left(\l(\f{3}{p}\r)-\frac23\right)\pmod {p^2},
\end{align*}
as desired.
\end{proof}

\medskip

\noindent{\it Proof of Theorem \ref{cat-main}}. In view of \eqref{binomialcon},
\begin{align*}
\sum_{k=0}^{(p-1)/2}\frac{\binom{2k}{k}}{(k+1)16^k}T_{2k}(4,1)&\eq \CT_x\sum_{k=0}^{m}\f{(-1)^k}{k+1}\binom{m+k}{2k}(4+x+x^{-1})^{2k}\\
&=\CT_xR_m(v)\pmod{p^2}.
\end{align*}
By Lemma \ref{CTxR}, we have
$$
\CT_xR_m(v)=(-1)^m\l(\f{d_{2m+2}}{m+1}+\f{d_{2m}}{m}\r)\eq 2(-1)^m\l((d_{p+1}-d_{p-1})-p(d_{p+1}+d_{p-1})\r)\pmod{p^2}.
$$
With the help of Lemmas \ref{dngfunc}, \ref{lem:dplus} and \ref{lem:dminus}, we obtain
\begin{align*}
d_{p+1}+d_{p-1}=[t^{p-1}]\f{1+t}{(1+t^2)\sqrt{1-6t+t^2}}&\eq\l(\f{3}{p}\r)\pmod{p},\\
d_{p+1}-d_{p-1}=[t^{p-1}]\f{(1-t^2)(1+t)}{(1+t^2)^2\sqrt{1-6t+t^2}}&\eq\frac23\l(\f{-3}{p}\r)+p\left(\l(\f{3}{p}\r)-\frac23\right)\pmod {p^2}.
\end{align*}
Combining the above, we finally arrive at
$$
\sum_{k=0}^{(p-1)/2}\frac{\binom{2k}{k}}{(k+1)16^k}T_{2k}(4,1)\eq\frac43\left(\l(\frac3p\r)-p\l(\frac{-1}{p}\r)\right)\pmod {p^2}.
$$

The proof of Theorem \ref{cat-main} is now complete.\qed

\section{Proof of Theorem \ref{triple-main}}\label{sec:triple}

In this section, we use $\CT_{x,y}F(x,y)$ to denote the constant term of a Laurent series or rational function $F(x,y)$.

The following result is a congruence analogue of Lemma \ref{CTid}.
\begin{lemma}\label{CTid-analog}
For any odd prime $p$, let $\alpha(t)$ and $\beta(t)$ be polynomials in $t$ with coefficients over $\Z_p$ such that $\alpha(0)\not\eq0\pmod{p}$. Then
\[
\operatorname{CT}_x(\alpha(t)+\beta(t)(x+x^{-1}))^{p-1}\eq(\alpha(t)^2-4\beta(t)^2)^m\pmod{p}.
\]
\end{lemma}

\begin{proof}
Clearly, by \eqref{centralbinomial},
\begin{align*}
\CT_x(\alpha(t)+\beta(t)(x+x^{-1}))^{p-1}&=\CT_x\sum_{k=0}^{p-1}\binom{p-1}{k}\beta(t)^k(x+x^{-1})^k\alpha(t)^{p-1-k}\\
&=\sum_{k=0}^{m}\binom{p-1}{2k}\binom{2k}{k}\beta(t)^{2k}\alpha(t)^{p-1-2k}\\
&\eq\alpha(t)^{p-1}\sum_{k=0}^m\binom{m}{k}\l(\f{-4\beta(t)^2}{\alpha(t)^2}\r)^k\\
&=\alpha(t)^{p-1}\l(1-\f{4\beta(t)^2}{\alpha(t)^2}\r)^m\\
&=(\alpha(t)^2-4\beta(t)^2)^m\pmod{p}.
\end{align*}
This concludes the proof.
\end{proof}

\begin{lemma}\label{coeff}
For any odd prime $p$, let $f(t)=\sum_{k=0}^{\infty}a_kt^k$ with $a_k\in\Z_p$. Then
\begin{align*}
[t^{p-1}](t-1)^{p-1}f(t)&\eq \sum_{k=0}^{p-1}a_k\pmod{p},\\
[t^{2p-1}](t-1)^{p-1}f(t)&\eq \sum_{k=p}^{2p-1}a_k\pmod{p}.
\end{align*}
\end{lemma}

\begin{proof}
Note that
$$
(t-1)^{p-1}=\f{(t-1)^p}{t-1}\eq \f{t^p-1}{t-1}=\sum_{k=0}^{p-1}t^k\pmod{p}.
$$
Therefore,
\begin{align*}
[t^{p-1}](t-1)^{p-1}f(t)&\eq [t^{p-1}]\sum_{k=0}^{p-1}t^k\sum_{l=0}^{p-1}a_lt^l\\
&=\sum_{k=0}^{p-1}a_{p-1-k}\\
&=\sum_{k=0}^{p-1}a_{k}\pmod{p}.
\end{align*}
The second congruence can be deduced similarly.
\end{proof}

\noindent{\it Proof of Theorem \ref{triple-main}}. Note that 
$$
[y^n]\l(1+\f{y^3}{2}\r)^{n}=\begin{cases}\binom{3k}{k}/2^k,\quad&n=3k,\\[5pt] 0,\quad&\t{otherwise}\end{cases}
$$
and
$$
\binom{3k}{k}=\f{(3k)(3k-1)\cdots(2k+1)}{k!}\eq0\pmod{p}
$$
for $2p/3<k\leq p-1$. Therefore,
\begin{align}\label{triple-twoparts}
\sum_{k=0}^{p-1}\f{\binom{3k}{k}}{432^k}T_{3k}(6,1)&\eq\sum_{0\leq 3k<2p/3}\f{\binom{3k}{k}}{2^k}\f{T_{3k}(6,1)}{6^{3k}}\notag\\
&=\sum_{0\leq 3k<2p}[y^{3k}]\l(1+\f{y^3}{2}\r)^{3k}\f{T_{3k}(6,1)}{6^{3k}}\notag\\
&=\sum_{l=0}^{2p-1}[y^l]\l(1+\f{y^3}{2}\r)^{l}\f{T_l(6,1)}{6^l}\notag\\
&=\sigma_1+\sigma_2\pmod{p},
\end{align}
where
\begin{align*}
\sigma_1&=\sum_{l=0}^{p-1}[y^l]\l(1+\f{y^3}{2}\r)^{l}\f{T_l(6,1)}{6^l},\\
\sigma_2&=\sum_{l=0}^{p-1}[y^{p+l}]\l(1+\f{y^3}{2}\r)^{p+l}\f{T_{p+l}(6,1)}{6^{p+l}}.
\end{align*}

We first consider $\sigma_1\pmod{p}$. Clearly,
\begin{align}\label{sigma1reduce}
\sigma_1&=\CT_{x,y}\sum_{l=0}^{p-1}\l(\l(y^{-1}+\f{y^2}{2}\r)\l(\f{6+x+x^{-1}}{6}\r)\r)^l\notag\\
&=\CT_{x,y}\f{1-\l(\l(y^{-1}+\f{y^2}{2}\r)\l(\f{6+x+x^{-1}}{6}\r)\r)^p}{1-\l(y^{-1}+\f{y^2}{2}\r)\l(\f{6+x+x^{-1}}{6}\r)}\notag\\
&\eq [y^{p-1}]\CT_{x}\l(y-\l(1+\f{y^3}{2}\r)\l(\f{6+x+x^{-1}}{6}\r)\r)^{p-1}\notag\\
&\eq [y^{p-1}]\CT_x((12y-12-6y^3)-(2+y^3)(x+x^{-1}))^{p-1}\pmod{p}.
\end{align}
In view of Lemmas \ref{CTid-analog} and \ref{coeff}, we get
\begin{align}\label{sigma1eval}
\sigma_1&\eq [y^{p-1}]\l((12y-12-6y^3)^2-4(2+y^3)^2\r)^m\notag\\
&=[y^{p-1}]\l(16 (-1 + y)^2 (2 + y) (4 - 3 y + 2 y^3)\r)^m\notag\\
&\eq\sum_{k=0}^{p-1}c_k\pmod{p},
\end{align}
where 
$$c_k=[y^k](2 + y)^m (4 - 3 y + 2 y^3)^m.$$

Now we evaluate $\sigma_2\pmod{p}$. It is easy to see that
\begin{align*}
T_{p+l}(6,1)&=[x^{p+l}](x^2+6x+1)^{p+l}\\
&\eq [x^{p+l}](x^{2p}+6^px^p+1)(x^2+6x+1)^l\\
&=6^pT_l(6,1)\pmod{p}.
\end{align*}
Meanwhile, 
$$
[y^{p+l}]\l(1+\f{y^3}{2}\r)^{p+l}\eq [y^{p+l}]\l(1+\f{y^{3p}}{2^p}\r)\l(1+\f{y^3}{2}\r)^{l}=[y^{p+l}]\l(1+\f{y^3}{2}\r)^{l}\pmod{p}.
$$
Hence, via similar arguments as in \eqref{sigma1reduce} and \eqref{sigma1eval}, we have
\begin{align}\label{sigma2eval}
\sigma_2&\eq \sum_{l=0}^{p-1}[y^{p+l}]\l(1+\f{y^3}{2}\r)^{l}\f{T_l(6,1)}{6^l}\notag\\
&=[y^p]\CT_{x}\sum_{l=0}^{p-1}\l(\l(y^{-1}+\f{y^2}{2}\r)\l(\f{6+x+x^{-1}}{6}\r)\r)^l\notag\\
&\eq [y^{2p-1}](y-1)^{p-1}(y+2)^m(2y^3-3y+4)^m\notag\\
&\eq \sum_{k=p}^{2p-1}c_k\pmod{p}.
\end{align}

Substituting \eqref{sigma1eval} and \eqref{sigma2eval} into \eqref{triple-twoparts}, we have
$$
\sum_{k=0}^{p-1}\f{\binom{3k}{k}}{432^k}T_{3k}(6,1)\eq \sum_{k=0}^{2p-1}c_k\pmod{p}.
$$
Since $\deg((y+2)^m(2y^3-3y+4)^m)=4m=2p-2$, $\sum_{k=0}^{2p-1}c_k=\sum_{k=0}^{2p-2}c_k$ is exactly the value of $(y+2)^m(2y^3-3y+4)^m$ at $y=1$. Therefore,
$$
\sum_{k=0}^{p-1}\f{\binom{3k}{k}}{432^k}T_{3k}(6,1)\eq 3^{m}\cdot 3^m=3^{p-1}\eq1\pmod{p}.
$$
The proof of Theorem \ref{triple-main} is now complete.\qed

\begin{Acks}
This work is supported by the National Natural Science Foundation of China (grant 12201301) and the College Students' Innovation Training Projects of Nanjing Forestry University (grant 202610298075Z). The authors would like to thank Dr. Yu-Chen Sun at the University of Bristol for valuable comments and suggestions concerning Theorems 1.2 and 1.3.
\end{Acks}

\end{document}